\documentclass[11pt,reqno]{amsart}

\usepackage[T1]{fontenc}
\usepackage{lmodern}
\usepackage{microtype}
\usepackage{mathtools}
\usepackage{amssymb}
\usepackage{xcolor}
\usepackage[colorlinks=true,linkcolor=blue!55!black,citecolor=blue!55!black,urlcolor=blue!55!black]{hyperref}
\usepackage[nameinlink,capitalise,noabbrev]{cleveref}

\newtheorem{theorem}{Theorem}[section]
\newtheorem{proposition}[theorem]{Proposition}
\newtheorem{lemma}[theorem]{Lemma}
\theoremstyle{remark}
\newtheorem{remark}[theorem]{Remark}

\DeclareMathOperator{\Aut}{Aut}
\DeclareMathOperator{\Syl}{Syl}
\DeclareMathOperator{\pr}{pr}
\DeclareMathOperator{\Stab}{Stab}
\newcommand{\gen}[1]{\langle #1\rangle}
\newcommand{\GammaLA}{\Gamma_{L,A}}

\title[Subgroups of coprime index]{Generation of finite groups from subgroups of coprime index}
\author{Richie Sater}
\address{Independent Researcher, United States}
\email{richiesater@gmail.com}
\urladdr{https://orcid.org/0009-0007-9051-8207}
\date{August 12, 2026}

\begin{document}

\begin{abstract}
Let $d(G)$ denote the least size of a generating set of a finite group $G$.
We prove that if $G$ has a family $\mathcal H$ of subgroups such that
$d(H)\leq d$ for every $H\in\mathcal H$ and
\[
  \gcd\{\lvert G:H\rvert:H\in\mathcal H\}=1,
\]
then $d(G)\leq d+1$.  This gives an affirmative answer to Kourovka
Problem~21.87.  The proof reduces a minimal counterexample to a critical
crown-based power with nonabelian socle.  An exact crown multiplicity formula
and a uniform lower bound for conditional generation give a lower bound for
the number of crown factors.  A subgroup containing a Sylow $2$-subgroup gives
the contradictory upper bound, via a pointwise centralizer estimate for Sylow
$2$-subgroups of finite simple groups.
\end{abstract}

\maketitle

\section{Introduction}

For a finite group $X$, write $d(X)$ for the minimum cardinality of a
generating set of $X$.  Kourovka Problem~21.87, proposed by Lucchini, asks
whether $d(G)\leq d+1$ whenever $G$ has a nonempty family of subgroups, each
generated by at most $d$ elements, whose indices have greatest common divisor
$1$ \cite[Problem 21.87, p.~177]{kourovka21}.  Kov\'acs and Sim proved this
for soluble groups
\cite[Theorem~2]{KovacsSim1991}; Lucchini subsequently proved the general
$d+2$ bound \cite{Lucchini2000}.  We prove the sharp general bound.

\begin{theorem}\label{thm:main}
Let $d$ be a nonnegative integer, let $G$ be a finite group, and suppose that
$G$ has a nonempty family $\mathcal H$ of subgroups satisfying
\[
 d(H)\leq d\quad(H\in\mathcal H),
 \qquad
 \gcd\{\lvert G:H\rvert:H\in\mathcal H\}=1.
\]
Then $d(G)\leq d+1$.
\end{theorem}

The bound is best possible already when $d=1$.  Indeed, in $S_3$ the cyclic
subgroups $H_2=\gen{(12)}$ and $H_3=\gen{(123)}$ have
$\lvert S_3:H_2\rvert=3$ and $\lvert S_3:H_3\rvert=2$, while $d(S_3)=2$.

The proof combines six published inputs.  Lucchini's augmentation-ideal
bound gives the required module-generator estimate
\cite[Theorem~2 and Proposition~2]{Lucchini1992}.  Roggenkamp's
presentation-rank identity then shows that any counterexample must have
positive presentation rank \cite[Theorem~2.1]{Roggenkamp1979}.  A result of
Gruenberg eliminates soluble normal subgroups from a minimal counterexample
\cite[p.~218]{Gruenberg1976}; we use the explicit formulation in
\cite[2.1, p.~209]{Lucchini1990}.  Dalla Volta and Lucchini characterize
finite groups that require more generators than any of their proper
quotients as critical crown-based powers and determine the exact critical
multiplicity
\cite[Theorems~1.4 and 2.7]{DallaVoltaLucchini1998}.  Finally, Detomi and
Lucchini prove that the relevant conditional generation probability is at
least $1/2$ \cite{DetomiLucchini2013}.  The new ingredient is to use the
Sylow $2$-subgroup supplied by the index hypothesis to pack distinct
automorphism orbits inside one fiber.  Guralnick's theorem that a finite
nonabelian simple group is generated by a Sylow $2$-subgroup and an
involution \cite{Guralnick1986} supplies exactly the matching factor $1/2$.
The form of Guralnick's theorem used here is also stated explicitly in
\cite[Introduction]{BurnessGuralnick2024}.

The published inputs from Dalla Volta--Lucchini, Detomi--Lucchini, and
Guralnick depend on the classification of finite simple groups.  No
computational classification or finite enumeration is used in the proof.

\section{Preliminaries}

\subsection{The index hypothesis}

\begin{lemma}[Quotient inheritance]\label{lem:quotient}
If $N\unlhd G$, then $G/N$ satisfies the hypothesis of
\cref{thm:main} with the same value of $d$.
\end{lemma}

\begin{proof}
For $H\in\mathcal H$, the subgroup $HN/N$ is generated by at most $d$
elements and
\[
 \lvert G/N:HN/N\rvert=\lvert G:HN\rvert
 \quad\text{divides}\quad \lvert G:H\rvert.
\]
Consequently every common divisor of the image indices divides every
original index, and hence the image indices have greatest common divisor
one.
\end{proof}

\begin{lemma}[Sylow coverage]\label{lem:sylow-coverage}
The index gcd is one if and only if, for every prime $p$ dividing
$\lvert G\rvert$, some $H\in\mathcal H$ contains a Sylow $p$-subgroup of
$G$.
\end{lemma}

\begin{proof}
There is an $H\in\mathcal H$ for which
$p\nmid\lvert G:H\rvert$.  Thus the $p$-part of $\lvert H\rvert$ equals the
$p$-part of $\lvert G\rvert$, so every Sylow $p$-subgroup of $H$ is a Sylow
$p$-subgroup of $G$.  Conversely, if the index gcd were greater than one,
some prime $p\mid\lvert G\rvert$ would divide every index.  No member of
$\mathcal H$ could then contain a Sylow $p$-subgroup of $G$.
\end{proof}

Because all indices divide $\lvert G\rvert$, the gcd of an arbitrary family
of indices is already attained by a finite subfamily.  Thus no issue of an
infinite gcd is hidden in the statement, and the published augmentation-ideal
theorem may be applied to such a finite subfamily.

\begin{lemma}[Sylow transport]\label{lem:sylow-transport}
Let $P\in\Syl_p(X)$.
\begin{enumerate}
\item If $N\unlhd X$, then $P\cap N\in\Syl_p(N)$.
\item If $\pi:X\to Y$ is surjective, then $\pi(P)\in\Syl_p(Y)$.
\end{enumerate}
\end{lemma}

\begin{proof}
The first assertion is the standard normal-subgroup part of the Sylow
theorems.  Applying it to $\ker\pi$ gives
\[
 |\pi(P)|=\frac{|P|}{|P\cap\ker\pi|}
          =\frac{|X|_p}{|\ker\pi|_p}=|Y|_p,
\]
which proves the second assertion.
\end{proof}

\subsection{Augmentation ideals and critical crowns}

Let $I_X$ denote the augmentation ideal of $\mathbb Z X$, and let
$d_X(I_X)$ be its minimum number of generators as a $\mathbb Z X$-module.
If $X$ is generated by $r$ elements, then $I_X$ is generated as a module by
the corresponding $r$ elements of the form $1-x$.  Lucchini's theorem
therefore yields the following consequence.

\begin{proposition}[Lucchini]\label{prop:augmentation}
Under the hypotheses of \cref{thm:main},
\[
 d_G(I_G)\leq d+1.
\]
\end{proposition}

This is \cite[Theorem~2 and Proposition~2]{Lucchini1992}.

\begin{proposition}[Roggenkamp]\label{prop:presentation-rank}
For every finite group $X$,
\[
 d(X)=d_X(I_X)+\pr(X),
\]
where $\pr(X)$ is the presentation rank of $X$.
\end{proposition}

Roggenkamp proves this in \cite[Theorem~2.1]{Roggenkamp1979}.
Lucchini also records the identity, with attribution to Roggenkamp, in
\cite[p.~146]{Lucchini1992}.

For later use, we record the soluble-normal-subgroup consequence of
Gruenberg's presentation-rank theory.

\begin{proposition}[Gruenberg]\label{prop:pr}
If $X$ is finite, $\pr(X)>0$, and $N\unlhd X$ is soluble, then
$d(X)=d(X/N)$.
\end{proposition}

We use the formulation in \cite[2.1, p.~209]{Lucchini1990}, which cites
\cite[p.~218]{Gruenberg1976}.

We next fix the crown notation.  Let $L$ be a finite monolithic group with
socle $A$.  When $A$ is abelian, assume that $A$ is complemented in $L$.
For $k\geq 1$, define the crown-based power
\[
 L_k=\{(\ell_1,\ldots,\ell_k)\in L^k:
          \ell_1A=\cdots=\ell_kA\}.
\]
Set
\[
 \GammaLA=
 \{\alpha\in\Aut(L):\alpha(\ell)A=\ell A\text{ for every }\ell\in L\}.
\]
Thus $\GammaLA$ is the group of automorphisms of $L$ inducing the identity
on $L/A$.

Let $\phi_X(m)$ be the number of ordered generating $m$-tuples of $X$.
For $m\geq d(L)$, put
\[
 P_{L,A}(m)=
 \frac{\phi_L(m)}{|A|^m\phi_{L/A}(m)}.
\]
This is the probability that an $m$-tuple of lifts generates $L$, conditional
on its image generating $L/A$.

We use the following two published results in exactly this form.

\begin{proposition}[Critical crown formula]\label{prop:crown}
Let $m\geq2$, and suppose that $d(X/N)\leq m$ for every nontrivial normal
subgroup $N$ of a finite group $X$, while $d(X)>m$.  Then
$X\cong L_{f_L(m)}$ for a monolithic group $L$ as above.  If its socle $A$
is nonabelian, then
\begin{equation}\label{eq:crown-formula}
 f_L(m)=1+\frac{P_{L,A}(m)|A|^m}{|\GammaLA|}.
\end{equation}
\end{proposition}

This is \cite[Theorems~1.4 and 2.7]{DallaVoltaLucchini1998}, with the
nonabelian branch of the latter normalized as follows.

In the notation of that source, the nonabelian-socle branch is
\[
 f_L(m)=1+\frac{\phi_L(m)}
 {|{\GammaLA}|\,\phi_{L/A}(m)}.
\]
Substitution of the preceding definition of $P_{L,A}(m)$ gives
\eqref{eq:crown-formula}; thus no probabilistic normalization is implicit in
the quoted formula.

\begin{proposition}[Detomi--Lucchini]\label{prop:probability}
If $L$ is a finite monolithic group with socle $A$ and $m\geq d(L)$, then
\begin{equation}\label{eq:half}
 P_{L,A}(m)\geq\frac12.
\end{equation}
\end{proposition}

This is the main inequality of \cite{DetomiLucchini2013}.

The notation in \cref{prop:crown} is chosen so that $f_L(m)$ is the first
crown multiplicity at which more than $m$ generators are needed.  In
particular, $f_L(m)\geq2$, and $L$ is a proper quotient of
$L_{f_L(m)}$.

\section{Pointwise Sylow centralizers}

For $Y\leq X$, write $C_{\Aut(X)}(Y)$ for the subgroup of automorphisms of
$X$ that fix every element of $Y$.

\begin{lemma}[Simple socle factor]\label{lem:simple-centralizer}
Let $S$ be a finite nonabelian simple group and let
$P\in\Syl_2(S)$.  Then
\begin{equation}\label{eq:simple-centralizer}
 |C_{\Aut(S)}(P)|\leq \frac{|S|}{2}.
\end{equation}
\end{lemma}

\begin{proof}
By Guralnick's theorem \cite{Guralnick1986}, there are
$P_0\in\Syl_2(S)$ and an involution $t_0\in S$ such that
$S=\gen{P_0,t_0}$.  Sylow conjugacy gives $s\in S$ with $P_0^s=P$;
after conjugating the generating pair and setting $t=t_0^s$, we have
$S=\gen{P,t}$.  Put $D=C_{\Aut(S)}(P)$.  The $S$-conjugacy class
$t^S$ is $D$-invariant.  Indeed, $t$ is conjugate in $S$ to an involution
$u\in P$; every element of $D$ fixes $u$, so it preserves $u^S=t^S$.

The stabilizer $D_t$ is trivial: an element of $D_t$ fixes $P$ pointwise and
fixes $t$, hence fixes the generating group $\gen{P,t}=S$ pointwise.  The
$D$-orbit of $t$ therefore has size $|D|$ and lies in $t^S$.  Since
$\lvert C_S(t)\rvert\geq2$,
\[
 |D|\leq |t^S|=\lvert S:C_S(t)\rvert\leq |S|/2.
\]
\end{proof}

\begin{lemma}[Monolithic group]\label{lem:monolithic-centralizer}
Let $L$ be monolithic with nonabelian socle $A$, and let
$R\in\Syl_2(L)$.  Then
\begin{equation}\label{eq:monolithic-centralizer}
 |C_{\GammaLA}(R)|\leq |A|/2.
\end{equation}
\end{lemma}

\begin{proof}
Write
$A=S_1\times\cdots\times S_n$, where the $S_i$ are isomorphic finite
nonabelian simple groups.  Restriction gives an embedding
$\GammaLA\hookrightarrow\Aut(A)$.  To see injectivity, suppose that
$\alpha\in\GammaLA$ fixes $A$ pointwise.  For $x\in L$, write
$\alpha(x)=xc$ with $c\in A$.  Comparing $\alpha(xax^{-1})$ with
$\alpha(x)\alpha(a)\alpha(x)^{-1}$ for $a\in A$ shows that
$c\in Z(A)=1$.  Hence $\alpha=1$.

Now
\[
 R\cap A=P_1\times\cdots\times P_n,
 \qquad P_i\in\Syl_2(S_i).
\]
Every $P_i$ is nontrivial.  An automorphism of $A$ which fixes
$R\cap A$ pointwise cannot permute the factors, because a nonidentity
element supported in $P_i$ would be moved to a different support.  By
\cref{lem:simple-centralizer}, restriction to the factors gives
\[
 |C_{\GammaLA}(R)|
 \leq \prod_{i=1}^n |C_{\Aut(S_i)}(P_i)|
 \leq (|S_1|/2)^n
 \leq |A|/2.
\]
\end{proof}

\section{Proof of the main theorem}

\begin{proof}[Proof of \cref{thm:main}]
If $d=0$, every member of $\mathcal H$ is trivial.  The gcd condition then
forces $|G|=1$, so assume $d\geq1$ and put $m=d+1\geq2$.

Suppose, for a contradiction, that the theorem is false.  For this fixed
value of $d$, choose a counterexample $G$ of least order.  By
\cref{lem:quotient}, every proper quotient of $G$ satisfies the same
hypothesis with the same $d$.  Minimality gives
\begin{equation}\label{eq:proper-quotients}
 d(G/N)\leq m
 \qquad(1\neq N\unlhd G),
\end{equation}
whereas $d(G)>m$.

By \cref{prop:augmentation}, $d_G(I_G)\leq m$, so
\cref{prop:presentation-rank} implies $\pr(G)>0$.  If $G$ had a nontrivial
soluble normal subgroup $N$, then \cref{prop:pr} and
\eqref{eq:proper-quotients} would give
$d(G)=d(G/N)\leq m$, a contradiction.  Thus
\begin{equation}\label{eq:no-soluble-normal}
 G\text{ has no nontrivial soluble normal subgroup.}
\end{equation}

Apply \cref{prop:crown}.  We have
\[
 G\cong L_k,
 \qquad k=f_L(m),
\]
where $L$ is monolithic with socle $A$.  The socle $A^k$ of $G$ is a
nontrivial normal subgroup, so \eqref{eq:no-soluble-normal} excludes the
abelian-socle case.  These are the two exhaustive socle branches in
\cref{prop:crown}, so \(A\) is nonabelian.  Also \(L\) is a proper quotient
of \(G\), so \(d(L)\leq m\).  Combining
\cref{prop:crown,prop:probability} gives
\begin{equation}\label{eq:lower-k}
 k
 =1+\frac{P_{L,A}(m)|A|^m}{|\GammaLA|}
 \geq 1+\frac{|A|^m}{2|\GammaLA|}.
\end{equation}

Each simple direct factor of $A$ contains the involution supplied by
Guralnick's theorem, so $A$, and hence $G$, has even order.  By
\cref{lem:sylow-coverage}, there is an $H\in\mathcal H$ containing a Sylow
$2$-subgroup $T$ of $G$.  After padding a generating tuple with identities,
write
\[
 H=\gen{h_1,\ldots,h_d},
 \qquad
 h_i=(h_{i1},\ldots,h_{ik})\in L_k.
\]
For $1\leq j\leq k$, set
\[
 \mathbf h_j=(h_{1j},\ldots,h_{dj})\in L^d.
\]
Let $\rho:L^d\to(L/A)^d$ be the coordinatewise quotient map.  All the tuples
$\mathbf h_j$ have the same image under $\rho$.  They therefore lie in a
common fiber $\Omega$, which has cardinality $|A|^d$.  The group $\GammaLA$
acts componentwise on $\Omega$.

We claim that $\mathbf h_1,\ldots,\mathbf h_k$ belong to distinct
$\GammaLA$-orbits.  Suppose instead that, for distinct $r,s$, there is an
$\alpha\in\GammaLA$ such that
\[
 h_{is}=\alpha(h_{ir})\qquad(1\leq i\leq d).
\]
It follows by evaluating words in the generators that every $h\in H$, written
as $h=(x_1,\ldots,x_k)$, satisfies $x_s=\alpha(x_r)$.  But, by
\cref{lem:sylow-transport}, $T\cap A^k$ is a Sylow $2$-subgroup of $A^k$,
hence is a direct product of $k$ nontrivial Sylow $2$-subgroups, one in each
coordinate.  Since $T\leq H$, choose $h\in T\cap A^k$ nonidentity and
supported only in coordinate $r$.  Then
$x_r\neq1=x_s$, contradicting $x_s=\alpha(x_r)$.  This proves the claim.

Let $T_j$ denote the image of $T$ under the $j$th coordinate projection
$L_k\to L$.  This projection is surjective, so
\cref{lem:sylow-transport} gives $T_j\in\Syl_2(L)$.  Put
$H_j=\gen{h_{1j},\ldots,h_{dj}}$, the image of $H$ under the $j$th coordinate
projection.  The stabilizer of $\mathbf h_j$ in $\GammaLA$ fixes $H_j$
pointwise.  Since $T\leq H$, we have $T_j\leq H_j$.  By
\cref{lem:monolithic-centralizer},
\[
 |\Stab_{\GammaLA}(\mathbf h_j)|\leq |A|/2.
\]
Thus each of the $k$ distinct orbits has at least
$2|\GammaLA|/|A|$ elements.  Packing them into $\Omega$ gives
\begin{equation}\label{eq:upper-k}
 k\frac{2|\GammaLA|}{|A|}\leq |A|^d,
 \qquad\text{and hence}\qquad
 k\leq \frac{|A|^{d+1}}{2|\GammaLA|}
   =\frac{|A|^m}{2|\GammaLA|}.
\end{equation}
The upper bound \eqref{eq:upper-k} contradicts the strict lower bound
\eqref{eq:lower-k}.  This proves $d(G)\leq d+1$.
\end{proof}

\begin{remark}
The sharp point is the comparison between \eqref{eq:lower-k} and
\eqref{eq:upper-k}.  The conditional probability supplies one factor
$1/2$, while the pointwise Sylow-centralizer estimate supplies the matching
factor in the orbit count.  The additional $1$ in the exact critical-crown
formula rules out equality.
\end{remark}

\section*{Research materials}

The manuscript source and source manifest are available in a public
\href{https://github.com/RichieSater/kourovka-21-87}{GitHub repository}.  An
archival snapshot is deposited at Zenodo under DOI
\href{https://doi.org/10.5281/zenodo.21893748}{10.5281/zenodo.21893748}.

\section*{Dependencies and disclosure}

The proof uses only the published results cited above and the arguments given
here; no computer calculation enters it.  The cited results on critical
crowns, conditional generation, and generation of simple groups depend on
the classification of finite simple groups.  A literature search conducted
before submission found no earlier paper asserting this conclusion, but this
does not establish priority.

OpenAI Codex was used for literature organization, proof stress-testing, and
editorial assistance.  The author is responsible for all content and
mathematical claims in this manuscript.

\bibliographystyle{amsplain}
\bibliography{references}

\end{document}